\documentclass[psamsfonts,12pt]{amsart}
\usepackage{graphicx}  
\usepackage{amsmath,amssymb,amsthm, amscd}
\usepackage{amssymb,fge}
\usepackage{tikz-cd}

\usepackage{color}

\DeclareMathOperator{\Dom}{Dom}

\DeclareMathOperator{\End}{End}

\DeclareMathOperator{\Orb}{Orb}

\newcommand{\mysetminus}{\mathbin{\fgebackslash}}

\newtheorem{theorem}{Theorem}[section]
\newtheorem{lemma}[theorem]{Lemma}

\newtheorem{proposition-definition}[theorem]{Proposition-Definition}
\newtheorem{corollary}[theorem]{Corollary}

\theoremstyle{definition}
\newtheorem*{definition}{Definition}
\theoremstyle{remark}
\newtheorem{remark}[theorem]{Remark}
\newtheorem{example}[theorem]{Example}

\usepackage[margin=1in]{geometry}  
\usepackage{graphicx}              
\usepackage{amsmath}               
\usepackage{amsfonts}              
\usepackage{amsthm}                
\usepackage{verbatim}              
\usepackage{indentfirst}           
\usepackage{underscore}
\usepackage{enumitem}
\usepackage{relsize}
\usepackage{varwidth}
\usepackage{mathtools} 
\usepackage{hyperref}
\hypersetup{
    colorlinks=true,
    linkcolor=blue,
    filecolor=magenta,      
    urlcolor=cyan,
}
\usepackage{amssymb}
\usepackage[all]{xy}
\makeatletter
\renewcommand*\env@matrix[1][*\c@MaxMatrixCols c]{%
  \hskip -\arraycolsep
  \let\@ifnextchar\new@ifnextchar
  \array{#1}}
\makeatother

\theoremstyle{remark}

\newcommand*{\Scale}[2][4]{\scalebox{#1}{$#2$}}%

\title[Dynamical Degrees of Random Maps]{Dynamical and arithmetic degrees for random iterations of maps on Projective Space}
\begin{document}
\author[Wade Hindes]{Wade Hindes}
\address{Department of Mathematics, Texas State University, 601 University Dr., San Marcos, TX 78666.}
\email{wmh33@txstate.edu}
\maketitle
\renewcommand{\thefootnote}{}
\begin{abstract} We show that the dynamical degree of an (i.i.d) random sequence of dominant, rational self-maps on projective space is almost surely constant. We then apply this result to height growth and height counting problems in random orbits.   
\end{abstract} 
\section{Introduction} 
Let $K$ be a number field and let $X$ be a projective variety defined over $K$. Given a collection $S$ of self maps on $X$, we are interested in the arithmetic properties of dynamical orbits generated by $S$. To define the orbits we consider, let $\gamma=(\theta_1,\theta_2\dots)$ be an infinite sequence of elements of $S$ and let $\gamma_n=\theta_n\circ\theta_{n-1}\circ\dots\circ\theta_1$ for $n\geq1$. Then given a point $P\in X$, we define the orbit of the pair $(\gamma,P)$ to be the set
\begin{equation}\label{orbitdef}
\Orb_\gamma(P)=\{P,\gamma_1(P),\gamma_2(P),\dots\}.
\end{equation} 
In particular, if $X$ is equipped with a height function $h:X(\overline{K})\rightarrow\mathbb{R}$, then we are interested in the growth rate of $h(\gamma_n(P))$ as we move within an orbit. This problem is (somewhat) understood when $S\subseteq\End(X)$ is a set of morphisms, as the height growth rate in orbits can be described by an associated canonical height function; see the seminal work of Call and Silverman \cite{Call-Silverman} for the case when $S=\{\phi\}$ is a single map and see work of Kawaguchi \cite{Kawaguchi} for a generalization to collections of maps. However, this problem is much less understood in the case of dominant rational maps, even for singletons $S$; see \cite{KawaguchiSilverman} for results along these lines. One obstruction that arises when the maps in $S$ fail to be morphisms is that the growth rate of the degree (say for maps on $\mathbb{P}^N$) is mysterious in general. Further still, when $S$ contains at least two maps (even morphisms), there exist some sequences of functions whose degree growth can be quite complicated, as well will see. 

To begin our study of dynamical and arithmetic degrees, we fix some notation. Let $X=\mathbb{P}^N$ and let $S\subseteq\Dom(\mathbb{P}^N)$ be a set of dominant rational maps on $\mathbb{P}^N$. Then (by analogy with the singleton case \cite{SilvermanPN}) we define the dynamical degree of infinite sequences:  
\begin{definition} Let $\gamma$ be a sequence of elements of $S$. Then the \emph{dynamical degree} of $\gamma$ is       
\begin{equation}\label{dyndeg}
\delta_\gamma=\lim_{n\rightarrow\infty}\deg(\gamma_n)^{1/n},
\end{equation} 
provided that the limit exists.
\end{definition}  
However, as was indicated above, the dynamical degree is more subtle in the case of multiple maps. In particular, the limit in (\ref{dyndeg}) need not exist in general, even for sets of morphisms. The following example was shown to the author by J. Silverman.      
\begin{example} Let $\phi_1,\phi_2:\mathbb{P}^N\rightarrow\mathbb{P}^N$ be morphisms of degree $d_1$ and $d_2$ respectively. Consider the set $S=\{\phi_1,\phi_2\}$ and the sequence $\gamma$ corresponding to  
\[\gamma\leftrightarrow(1,2,2,1,1,1,1,2,2,2,2,2,2,2,2,1,1,\dots);\]
here we alternately repeat $1$'s and $2$'s, each time doubling the number of occurrences. Then it is easy to see that  
\[\deg(\gamma_{2^k-1})=
\begin{cases} 
      d_1^{\frac{2}{3}2^k-\frac{1}{3}}d_2^{\frac{1}{3}2^k-\frac{2}{3}}& \text{if $k$ is odd,} \\[7pt]
      d_1^{\frac{1}{3}2^k-\frac{2}{3}}d_2^{\frac{2}{3}2^k-\frac{1}{3}} & \text{if $k$ is even.} \\
     \end{cases}
\]
Hence, it follows that 
\[\lim_{\substack{k\rightarrow\infty\\k\;\text{odd}}} \deg(\gamma_{2^k-1})^{1/(2^k-1)}=d_1^{2/3}d_2^{1/3}\;\;\;\text{and}\;\;\;\lim_{\substack{k\rightarrow\infty\\k\;\text{even}}} \deg(\gamma_{2^k-1})^{1/(2^k-1)}=d_1^{1/3}d_2^{2/3}.\vspace{.15cm}\] 
In particular, the limit defining $\delta_\gamma$ in (\ref{dyndeg}) does not exist (unless $d_1=d_2$). 
\end{example} 
On the other hand, one expects that the limit defining the dynamical degree exists for ``most" sequences. To make this precise we define a probability measure on the set of infinite sequences of elements of $S$. There are several ways to do this, but the most natural (from our perspective), is to fix a probability measure $\nu$ on the set $S$ and then extend this to the product space $\Phi_S=\Pi_{i=1}^\infty S$ via the product measure $\bar{\nu}$; see \cite[\S10.4]{ProbabilityText} for details. Probabilists refer to the corresponding probability space $(\Phi_S,\bar{\nu})$ as the space of independent, identically distributed (or i.i.d) sequences of elements of $S$ (distributed according to $\nu$). In particular, the measure $\bar{\nu}$ is an extension of the natural measure $\nu_n$ on the finite product $\Phi_{S,n}=S^n$, and we may therefore talk about the \emph{expected value} $\mathbb{E}_n[f_n]=\int_{\Phi_{S,n}} f_n\,d\nu_n$ of a random variable $f_n:\Phi_{S,n}\rightarrow\mathbb{R}$ on $n$-term sequences.
\begin{remark}{\label{discrete}} Since the sets $S$ we consider are always finite or countably infinite, we assume that all subsets of $S$ are $\nu$-measurable.  
\end{remark} 
We are ready to state our first result. That is, for suitable subsets $S\subseteq\Dom(\mathbb{P}^N)$ we show that the dynamical degree \textbf{exists and is constant} outside of a $\bar{\nu}$-measure zero set of sequences. This may be viewed as a sort of law of large numbers for random degree growth. In particular, the result below holds whenever $S$ is a finite set of dominant rational maps.          
\begin{theorem}{\label{thm:dyndeg}} Let $S\subseteq\Dom(\mathbb{P}^N)$ be a set of dominant rational self-maps on $\mathbb{P}^N$, let $\nu$ be a probability measure on $S$, and let $(\Phi_S,\bar{\nu})$ be the space of i.i.d sequences of elements of $S$ distributed according to $\nu$. If $\mathbb{E}_1[\log\deg(\phi)]$ exists, then there is a constant $\delta_{S,\nu}$, given by \vspace{.1cm} 
\[\delta_{S,\nu}:=\exp\Bigg(\inf_{n\geq1}\frac{\mathbb{E}_n\big[\log\,\deg(\gamma_n)\big]}{n} \Bigg),\] 
such that 
\[ \delta_{S,\nu}=\lim_{n\rightarrow\infty}\deg(\gamma_n)^{1/n}:=\delta_{\gamma}\vspace{.1cm}\] 
for $\bar{\nu}$-almost every sequence $\gamma\in\Phi_S$. 
\end{theorem} 
In particular, Theorem \ref{thm:dyndeg} motivates the following definition of the dynamical degree associated to a set of maps with a probability measure. 
\begin{definition} If $(S,\nu)$ satisfy the hypothesis of Theorem \ref{thm:dyndeg}, then we call the constant $\delta_{S,\nu}$ the dynamical degree of $(S,\nu)$.  
\end{definition} 
\begin{remark} See \cite{Mello} for an alternative definition of the dynamical degree of $S$ for finite sets. In particular, this definition uses the maximum degree of a map in $\Phi_{S,n}$ to define a corresponding $\limsup$ (not a full limit). However, we note that if $S=\{\phi_2,\phi_2\}$ are maps of degree $d_1$ and $d_2$ (with $d_1<d_2$) then the \emph{only} sequence $\gamma$ that obtains this maximum is the constant sequence associated to $\phi_2$. For this reason, we believe that the constant $\delta_{S,\nu}$ from Theorem \ref{thm:dyndeg}, subject to some weighting of sequences, is a more reflective measure of the collective degree growth of arbitrary compositions of maps in $S$.  
\end{remark}   
On the other hand, in the case of morphisms, we are able to calculate explicitly the dynamical degree of a pair $(S,\nu)$, especially in the case of finite sets of morphisms.   
\begin{theorem}\label{thm:morphdyndeg} Let $S\subseteq\End(\mathbb{P}^N)$ be a set of non-constant endomorphisms on $\mathbb{P}^N$, let $\nu$ be a probability measure on $S$, and let $(\Phi_S,\bar{\nu})$ be the space of i.i.d sequences of elements of $S$ distributed according to $\nu$. Then the dynamical degree $\delta_{S,\nu}$ from Theorem \ref{thm:dyndeg} is given by \vspace{.1cm}   
\[\delta_{S,\nu}=\exp\Big(\mathbb{E}_1\big[\log\deg(\phi)\big]\Big).\] 
In particular, if $S$ is finite, then 
\[\delta_{S,\nu}=\prod_{\phi\in S}\deg(\phi)^{\nu(\phi)},\]
i.e., $\delta_{\nu,S}$ is a weighted geometric mean of the degrees of the maps in $S$. 
\end{theorem}
\begin{remark} Likewise in the case of morphisms (of unequal degrees), one may formulate a central limit theorem for degree growth. To formulate such a statement, define the i.i.d random variables $X_i:\Phi_S\rightarrow\mathbb{R}$ given by $X_i(\gamma)=\log\deg(\theta_i)$ for $\gamma=(\theta_n)_{n\geq1}$. Moreover, assume that the variance $\sigma^2$ of the $X_i$ is non-zero. Then when $S$ is a set of morphisms, \vspace{.05cm}   
\[Z_n(\gamma)=\frac{\sqrt{n}}{\sigma}\Bigg(\frac{X_1(\gamma)+\dots +X_n(\gamma)}{n}- \log\delta_{S,\nu}\Bigg)=\frac{\sqrt{n}}{\sigma}\Big(\log\deg(\gamma_n)^{1/n} - \log\delta_{S,\nu}\Big).\] 
Moreover, the Central Limit Theorem \cite[Theorem 3.4.1]{durrett} implies that $Z_n$ converges in distribution to the standard normal distribution; see also \cite[Theorem 21.4]{ProbabilityText} for an explicit convergence rate estimate for the Central Limit Theorem.      
\end{remark}
Equipped with some understanding of dynamical degrees, we now turn our attention to the study of height growth rates in orbits. However, for rational maps (with non-trivial indeterminacy locus) we must take some care to make sure that the orbits we consider are well-defined. Given a dominant rational map $\phi: \mathbb{P}^N(\overline{\mathbb{Q}})\dashrightarrow\mathbb{P}^N(\overline{\mathbb{Q}})$ in $S$, we let $I_\phi$ be the indeterminacy locus of $\phi$ (i.e., $I_\phi$ is the set of points where $\phi$ is not defined). Similarly, let $I_S=\cup_{\phi\in S} I_\phi$. Then, we restrict ourselves to initial points in 
\[\mathbb{P}^N(\overline{\mathbb{Q}})_S:=\{P\in\mathbb{P}^N(\overline{\mathbb{Q}})\,:\, \Orb_\gamma(P)\cap I_S=\varnothing\;\text{for all $\gamma\in\Phi_S$}\}.\]
Moreover, in order to control certain height estimates in orbits (which become unruly when maps of degree $1$ appear), we need the following condition on $S$. 
\begin{definition} A set of dominant rational maps $S$ on $\mathbb{P}^N$ is called \emph{degree independent} if $\deg(\gamma_k)\geq2$ for all $\gamma_k\in\Phi_{S,k}$ and all $k\geq1$.   
\end{definition}
\begin{remark} If $S$ is a set of morphisms of degree at least $2$, then $S$ is degree independent. Therefore, the notion of degree independence is only necessary for non-morphisms.  
\end{remark} 
With the above definitions in place, we are able to bound the height growth rate of orbits (almost surely) by the dynamical degree; compare to the case when $S$ is a single rational map in \cite{KawaguchiSilverman,SilvermanPN}.      
\begin{theorem}{\label{thm:arith<dyn}} Let $S$ be a finite set of degree independent rational maps on $\mathbb{P}^N$, let $\nu$ be a probability measure on $S$, and let $(\Phi_S,\bar{\nu})$ be the space of i.i.d sequences of elements of $S$ distributed according to $\nu$. Then for all $P\in \mathbb{P}^N(\overline{\mathbb{Q}})_S$, we have that \vspace{.1cm}  
\[\limsup_{n\rightarrow\infty} h(\gamma_n(P))^{1/n}\leq \delta_{S,\nu}\]
for almost every $\gamma\in\Phi_S$.    
\end{theorem} 

Of course, one would like to replace the $\limsup$ in Theorem \ref{thm:arith<dyn} with a limit and the upper bound with an equality. Likewise, one would like to obtain results for (suitable) infinite sets $S$, as we did for dynamical degrees. To do this, we restrict our attention to sets of morphisms. However, as we hope to illustrate below, this remains a subtle problem even for morphisms; see Example \ref{eg:halffinite}. Moreover, to alleviate some of the difficulty that the aforementioned example presents, we assume further that our initial point $P\in\mathbb{P}^N$ is well-behaved in the following sense. 
\begin{definition} For $P\in\mathbb{P}^N(\overline{\mathbb{Q}})_S$, we say that $P$ is \emph{almost surely wandering} for $(S,\nu)$ if the set of sequences  
\[\{\gamma\in\Phi_S:\,\Orb_\gamma(P)\; \text{is finite}\}\]
has $\bar{\nu}$-measure $0$.     
\end{definition}
\begin{remark} If $S$ is a set of morphisms of degree at least $2$, then the set of points that fail to be almost surely wandering has bounded height; see \cite[Theorem 1.1]{Kawaguchi}. In particular over a fixed number field, there are at most finitely many points that can fail to be almost surely wandering.     
\end{remark}  
We illustrate the difficulty of tracking the height growth rate in orbits when the initial point fails to be almost surely wandering.   
\begin{example}{\label{eg:halffinite}} Let $S=\{\phi_1,\phi_2\}=\{2x^2,x^2+x\}$, let $\nu(2x^2)=1/2=\nu(x^2+x)$, and let $P=-1$. Then it is straightforward to check that 
\[\bar{\nu}\big(\{\gamma\in\Phi_S\,:\, \text{$\Orb_\gamma(P)$ is finite}\}\big)=1/2.\]
Hence, $P$ is not almost surely wandering for $(S,\nu)$. Moreover, only ``half" of all orbits have a chance to reach the upper bound in Theorem \ref{thm:arith<dyn}. The obstruction here is that $\phi_2(P)=0$ has finite orbit for every $\gamma\in\Phi_S$. In Theorem \ref{thm:a.s} below, we show that this type of obstruction completely characterizes how an initial point can fail to be almost surely wandering.        
\end{example}
With these notions in place, we are ready to define the arithmetic degree of an orbit; c.f.  \cite{KawaguchiSilverman,SilvermanPN}. Specifically, since one expects (generically) that $h(\gamma_n(P))$ exhibits exponential growth, we make the following definition.   
\begin{definition} Let $\gamma\in\Phi_S$ and let $P\in\mathbb{P}^N(\overline{\mathbb{Q}})_S$. Then we define the \emph{arithmetic degree} of the pair $(\gamma,P)$ to be the limit, 
\[ \alpha_\gamma(P):=\lim_{n\rightarrow\infty} h(\gamma_n(P))^{1/n},\]
provided that the limit exists; here $h:\mathbb{P}^N(\overline{\mathbb{Q}})\rightarrow\mathbb{R}$ is the (absolute) Weil height \cite[\S3.1]{SilvDyn}. 
\end{definition} 
By analogy with the well-known case of a single morphism, we will prove that the arithmetic degree of a pair $(\gamma,P)$ equals the dynamical degree $\delta_{S,\nu}$ defined in Theorem \ref{thm:dyndeg} with $\bar{\nu}$-probability one. The proof of this result uses the notion of canonical heights constructed in \cite{Kawaguchi} and revisited in \cite{stochastic}. However, in order to define these canonical heights we must stipulate that $S$ have further properties, which we now discuss. 
\begin{remark} In Section \ref{sec:arithdeg}, we use canonical heights for infinite sequences of endomorphisms on $\mathbb{P}^N$. However, it is worth mentioning that the constructions in \cite{stochastic,Kawaguchi} that we use in Section \ref{sec:arithdeg} work for polarizable sets of endomorphisms on more general projective varieties.   
\end{remark} 
It is well known that if $\phi:\mathbb{P}^N(\overline{\mathbb{Q}})\rightarrow\mathbb{P}^N(\overline{\mathbb{Q}})$ is a morphism of degree $d_\phi$, then 
\begin{equation}\label{functoriality}
h(\phi(P))=d_\phi h(P)+O_{\phi}(1)\;\;\;\text{for all $P\in\mathbb{P}^N(\overline{\mathbb{Q}})$;} \vspace{.1cm} 
\end{equation}  
see, for instance, \cite[Theorem 3.11]{SilvDyn}. With this in mind, we let 
\begin{equation}{\label{htconstant}} 
C(\phi):=\sup_{P \in \mathbb{P}^N(\bar{\mathbb{Q}})} \Big\vert h(\phi(P))-d_\phi h(P)\Big\vert 
\end{equation} 
be the smallest constant needed for the bound in (\ref{functoriality}). Then, in order to generalize the construction of canonical heights for a single map to a collection of maps (equivalently, from constant sequences to arbitrary sequences), we define the following fundamental notion.  
\begin{definition} 
A set of endomorphisms $S$ on $\mathbb{P}^N$ is called \emph{height controlled} if the following properties hold: \vspace{.1cm} 
\begin{enumerate} 
\item $d_\phi\geq2$ for all $\phi\in S$. \vspace{.1cm}
\item $\sup_{\phi\in S}\, C(\phi)$ is finite. \vspace{.1cm}
\end{enumerate} 
\end{definition} 

\begin{example}[Finite collections] Let $S=\{\phi_1,\phi_2,\dots,\phi_s\}$ be any finite collection of endomorphisms on $\mathbb{P}^N$ of degree at least two. Then $S$ is height controlled.  
\end{example}  
\begin{example}[Twisted power maps]{\label{unicrit}} For any non-constant, finite set of rational maps $S$ on $\mathbb{P}^1$, the set $\bar{S}=\{\phi\circ x^d\,: \phi\in S,\, d\geq2\}$ is height controlled (and infinite). 
\end{example}
\begin{example} Let $\mathcal{U}$ be the set of roots of unity in $\overline{\mathbb{Q}}$. Then the set $S=\{x^2+u\,:\, u\in \mathcal{U}\}$ is a height controlled collection of maps on $\mathbb{P}^1$. Moreover, it is worth pointing out that $S$ has a corresponding probability measure given by embedding $\mathcal{U}$ in the unit circle (in $\mathbb{C}$) and then taking the Haar measure on the circle.  
\end{example} 
We are now ready to relate arithmetic and dynamical degrees for height controlled sets of morphisms. 
\begin{theorem}\label{thm:arithdeg} Let $S\subseteq\End(\mathbb{P}^N)$ be a set of endomorphisms on $\mathbb{P}^N$, let $\nu$ be a probability measure on $S$, and let $(\Phi_S,\bar{\nu})$ be the space of i.i.d sequences of elements of $S$ distributed according to $\nu$. Moreover, assume that $S$ has the following properties: \vspace{.1cm} 
\begin{enumerate} 
\item[\textup{(1)}] $S$ is height controlled. \vspace{.15cm} 
\item[\textup{(2)}] $\mathbb{E}_1[\log\deg(\phi)]$ exists. \vspace{.1cm} 
\end{enumerate} 
Then if $P\in\mathbb{P}^N(\overline{\mathbb{Q}})$ is an almost surely wandering point for $(S,\nu)$, we have that
\[\alpha_\gamma(P)=\delta_{S,\nu}\]
for $\bar{\nu}$-almost every $\gamma\in\Phi_S$.  
\end{theorem}   
Tying together Theorems \ref{thm:morphdyndeg} and \ref{thm:arithdeg} above, we prove that if $(S,\nu)$ is a height-controlled collection of morphisms (a sufficient condition to define suitable canonical heights) and $P$ is an almost surely wandering point for $(S,\nu)$, then the height of most points in almost all orbits grow in a uniform way depending on the dynamical degree. The result below may be viewed as a generalization of \cite[Proposition 3]{KawaguchiSilverman}.         
\begin{corollary}\label{cor:arithdeg} Let $S\subseteq\End(\mathbb{P}^N)$ be a set of endomorphisms on $\mathbb{P}^N$, let $\nu$ be a probability measure on $S$, and let $(\Phi_S,\bar{\nu})$ be the space of i.i.d sequences of elements of $S$ distributed according to $\nu$. Moreover, assume that $S$ has the following properties: \vspace{.1cm} 
\begin{enumerate} 
\item[\textup{(1)}] $S$ is height controlled. \vspace{.1cm} 
\item[\textup{(2)}] $\mathbb{E}_1[\log\deg(\phi)]$ exists. \vspace{.1cm} 
\end{enumerate}  
Then if $P\in\mathbb{P}^N(\overline{\mathbb{Q}})$ is an almost surely wandering point for $(S,\nu)$, we have that \vspace{.1cm} 
\[\lim_{B\rightarrow\infty}\frac{\#\{n:h(\gamma_n(P))\leq B\}}{\log(B)}=\frac{1}{\delta_{S,\nu}}\]
for $\bar{\nu}$-almost every $\gamma\in\Phi_S$. In particular, if $S$ is finite, then \vspace{.1cm} 
\[\lim_{B\rightarrow\infty}\frac{\#\{n:h(\gamma_n(P))\leq B\}}{\log(B)}=\prod_{\phi\in S}\deg(\phi)^{-\nu(\phi)},\]  
for $\bar{\nu}$-almost every $\gamma\in\Phi_S$.  
\end{corollary} 
\begin{example} Let $c\in\overline{\mathbb{Q}}$ be any fixed constant, let $S=\big\{x^d+c: d\geq2\big\}$, and let $\nu$ be the probability measure on $S$ induced by $\nu(x^d+c)=\frac{1}{e(d-2)!}$. Then $S$ is a height controlled collection of morphisms on $\mathbb{P}^1$ and \[\mathbb{E}_1[\log\deg(\phi)]=\sum_{d=2}^\infty\frac{\log d}{e(d-2)!}\]
is a convergent series. Therefore, the dynamical degree of $(S,\nu)$ exists; in fact, $\delta_{S,\nu}\approx2.85052$. In particular, Corollary \ref{cor:arithdeg} implies that for all but finitely many points $P\in\mathbb{P}^1(\mathbb{Q}(c))$ and $\bar{\nu}$-almost every sequence $\gamma\in\Phi_S$, the following estimate holds: 
\[\lim_{B\rightarrow\infty}\frac{\#\{n:h(\gamma_n(P))\leq B\}}{\log B}\approx\frac{1}{2.85052}\]  
\end{example} 
Finally, we give an alternative characterization of almost surely wandering points using the expected canonical height function $\mathbb{E}_{\bar{\nu}}[\hat{h}]:\mathbb{P}^N\rightarrow\mathbb{R}_{\geq0}$ constructed in \cite[Theorem 1.2]{stochastic} and a characterization of its kernel \cite[Corllary 1.4]{stochastic}); see \cite{stochastic} or Section \ref{sec:aswandering} for more details. To state this recharacterization, we define the (total) $S$-orbit of a point $Q\in\mathbb{P}^N(\overline{\mathbb{Q}})_S$ to be the union of all possible orbits defined in (\ref{orbitdef}), namely 
\begin{equation}\label{Sorbit}  
\Orb_S(Q):=\bigcup_{\gamma\in\Phi_S}\Orb_\gamma(Q).
\end{equation} 
In particular, we say that $Q$ has finite $S$-orbit if $\Orb_S(Q)$ is a finite set.       
\begin{theorem}\label{thm:a.s} Let $S$ be a finite set of endomorphisms on $\mathbb{P}^N$ all of degree at least $2$, let $\nu$ be a probability measure on $S$, and let $(\Phi_S,\bar{\nu})$ be the space of i.i.d sequences of elements of $S$ distributed according to $\nu$. Then $P\in\mathbb{P}^N(\overline{\mathbb{Q}})$ is almost surely wandering for $(S,\nu)$ if and only if $\mathbb{E}_{\bar{\nu}}[\hat{h}](Q)>0$ for all $Q\in\Orb_S(P)$.  
\end{theorem}  
\begin{remark} Equivalently, Theorem \ref{thm:a.s} states: $P\in \mathbb{P}^N(\overline{\mathbb{Q}})$ is \emph{not} almost surely wandering for $(S,\nu)$ if and only if there exists $Q\in\Orb_S(P)$ such that $Q$ has finite $S$-orbit.
\end{remark}   
\begin{remark} It is possible that some of the results in this paper hold for more general projective varieties, perhaps by using the submultiplicativity of degrees (up to a constant) recently established in \cite{submult}, Kingman's Ergodic Theorem \cite{kingman}, and techniques for estimating heights developed in \cite{KawaguchiSilverman}. This is the subject of current work by the author.        
\end{remark} 
\textbf{Acknowledgements:} We thank Vivian Olsiewski Healey, Shu Kawaguchi, and Joseph Silverman for discussions related to the work in this paper. 
\section{Dynamical Degrees}
When $S=\{\phi\}$ is a single map, a crucial tool for establishing the convergence of the limit defining the dynamical degree is Fekete's lemma (see the proof of \cite[Proposition 7]{SilvermanPN}), which states that if $a_n$ is a subadditive sequence of non-negative real numbers, then $\lim a_n/n$ exists. In particular, the following landmark theorem due to Kingman \cite{kingman} may be viewed as a random version of Fekete's lemma. In what follows, the expected value of a random variable $f$ on a probability space $\Omega$ is the integral $\int_\Omega f$.      
\begin{theorem}[Kingman's Subadditive Ergodic Theorem]\label{thm:kingman} Let $T$ be a measure preserving transformation on a probability space $(\Omega,\Sigma, \mu)$, and let $(g_n)_{n\geq1}$ be a sequence of $L^1$ random variables that satisfy the subadditivity relation  
\begin{equation}\label{subadd} 
g_{m+n}\leq g_n+g_m\circ T_n   
\end{equation} 
for all $n,m\geq1$. Then there exists a $T$-invariant function $g$ such that  
\[\lim_{n\rightarrow\infty}\frac{g_n(x)}{n}=g(x)\]  
for $\mu$-almost every $x\in\Omega$. Moreover, if $T$ is ergodic, then $g$ is constant and
\[\lim_{n\rightarrow\infty}\frac{g_n(x)}{n}=\lim_{n\rightarrow\infty}\frac{\mathbb{E}[g_n]}{n} =\inf_{n\geq1}\frac{\mathbb{E}[g_n]}{n}.\]    
for $\mu$-almost every $x\in\Omega$
\end{theorem}
\begin{remark} 
See \cite{kingman} for Kingman's original statement, see \cite[\S7.4]{durrett} for an exposition and a discussion of later improvements, and see \cite{steve,steele} for the version of Kingman's theorem that we use here.   
\end{remark}
To apply Kingman's Subadditive Ergodic Theorem to dynamical degrees, we use the following well known example of an ergodic, measure preserving transformation. In particular, the lemma below is a simple consequence of Kolmogorov's $0$\,-$1$ law \cite[Theorem 10.6]{ProbabilityText}; for a further discussion, see \cite[Example 7.1.6]{durrett} or \cite[Example 5.5]{steve2} and \cite[Exercise 5.11]{steve2}.    
\begin{lemma}\label{shift} Let $S$ be a set with probability measure $\nu$, and let $(\Phi_S,\bar{\nu})$ be the probability space of i.i.d sequences of elements of $S$ distributed according to $\nu$. Then the shift map, 
\[T\big((\theta_j)_{j\geq1}\big)=(\theta_j)_{j\geq2},\] 
is an ergodic, measure preserving transformation.   
\end{lemma}
\begin{proof}[(Proof of Theorem \ref{thm:dyndeg})] For $n\geq1$, consider the random variables $g_n:\Phi_S\rightarrow\mathbb{R}_{\geq0}$ given by 
\[g_n(\gamma):=\log\deg(\gamma_n);\] 
note that $g_n$ is non-negative since $S$ is a collection of dominant maps. We will show that the sequence $(g_n)_{n\geq1}$ satisfies the hypothesis of Kingman's Subadditive Ergodic Theorem. First note that each $g_n$ is $\bar{\nu}$-measurable since $g_n$ factors through the finite product $\Phi_{S,n}$, and every function on $\Phi_{S,n}:=S^n$, a countable set, is measurable; see Remark \ref{discrete} above and \cite[Theorem 10.4]{ProbabilityText} for standard facts about product measures. On the other hand, define $f_i:\Phi_S\rightarrow\mathbb{R}_{\geq0}$ given by $f_i(\gamma)=\log\deg(\theta_i)$ for $\gamma=(\theta_s)_{s=1}^\infty$. Then $g_n\leq\sum_{i=1}^nf_i$, since 
\begin{equation}\label{degbd} 
\deg(F\circ G)\leq\deg(F)\deg(G)\;\;\;\;\text{for any}\; F,G\in \Dom(\mathbb{P}^N).
\end{equation} 
In particular,   
\[\mathbb{E}[g_n]\leq\sum_{i=1}^n\mathbb{E}[f_i]=n\,\mathbb{E}[f_1]:=n\,\mathbb{E}_1[\log\deg(\phi)];\] 
here we use that $\Phi_S$ consists of i.i.d sequences. In particular, each $g_n$ is an $L^1$ function since $\mathbb{E}_1[\log\deg(\phi)]$ is bounded by assumption. Now we check the subadditivity relation in (\ref{subadd}), a simple consequence of (\ref{degbd}). Let $n,m>0$, let $\gamma=(\theta_s)_{s=1}^\infty$, and let $T$ be the shift map on $\Phi_S$. Then we compute that \vspace{.25cm}  
\begin{equation*}
\begin{split} 
g_{n+m}(\gamma)=\log\deg(\theta_{m+n}\circ\dots\circ\theta_1)&\leq\log\deg(\theta_{m+n}\circ\dots\circ\theta_{n+1})+\log\deg(\theta_n\circ\dots\circ\theta_1)\\[3pt] 
&=g_m(T^n(\gamma))+g_n(\gamma), \vspace{.25cm}
\end{split} 
\end{equation*}    
by (\ref{degbd}). In particular, Theorem \ref{thm:kingman} and Lemma \ref{shift} together imply that \vspace{.2cm} 
\begin{equation}\label{kinglim} 
\lim_{n\rightarrow\infty}\log\deg(\gamma_n)^{1/n}=\lim_{n\rightarrow\infty}\frac{g_n(\gamma)}{n}=\lim_{n\rightarrow\infty}\frac{\mathbb{E}[g_n]}{n}=\inf_{n\geq1}\frac{\mathbb{E}[g_n]}{n}=\inf_{n\geq1}\frac{\mathbb{E}_n[\log\deg(\gamma_n)]}{n} \vspace{.2cm}
\end{equation} 
for $\bar{\nu}$-almost every $\gamma\in\Phi_S$. Moreover, applying the exponential map to this equality and exchanging $\exp$ with the limit (justified, by continuity) gives the desire expression for $\delta_{S,\nu}$. 
\end{proof}  
It is worth mentioning that the case of morphisms in Theorem \ref{thm:morphdyndeg} follows from an older ergodic theorem due to Birkhoff; see, for instance, \cite[Theorem 7.2.1]{durrett} or \cite[Theorem 5.12]{steve2}. 
\begin{theorem}[Birkhoff's Ergodic Theorem]\label{birk} If $T$ is an ergodic, measure preserving transformation on a probability space $(\Omega,\Sigma, \mu)$, then for every random variable $f\in L^1(\Omega)$,
\begin{equation}\label{birkhoff} \lim_{n\rightarrow\infty} \frac{1}{n}\sum_{j=0}^{n-1} f\circ T^j(x)=\mathbb{E}[f].  
\end{equation}    
for $\mu$-almost every $x\in\Omega$. 
\end{theorem}
\begin{proof}[(Proof of Theorem \ref{thm:morphdyndeg})] Let $T$ be the shift map on $\Phi_S$ and let $f:\Phi_S\rightarrow\mathbb{R}_{\geq0}$ be given by $f(\gamma)=\log\deg(\theta_1)$ for $\gamma=(\theta_s)_{s\geq1}$. Then Lemma \ref{shift} and Birkhoff's Ergodic Theorem imply that 
\begin{equation*}
\begin{split} 
\lim_{n\rightarrow\infty}\log \deg(\gamma_n)^{1/n}=\lim_{n\rightarrow\infty}\frac{\log\big(\prod_{j=1}^n\deg(\theta_j)\big)}{n}&=\lim_{n\rightarrow\infty}\frac{1}{n}\sum_{j=1}^n\log\deg(\theta_j) \\[4pt] 
&=\lim_{n\rightarrow\infty}\frac{1}{n}\sum_{j=0}^{n-1}f\circ T^j(\gamma)=\mathbb{E}[f]=\mathbb{E}_1[\log\deg(\phi)]
\end{split} 
\end{equation*} 
for $\bar{\nu}$-almost every $\gamma\in\Phi_S$; here we use that the degree of the composition of morphisms is the product of their corresponding degrees. 
\end{proof}  
\begin{remark}{\label{rmk:lln}} An alternate proof of Theorem \ref{thm:morphdyndeg} can be given using the strong law of large numbers \cite[Theorem 2.5.6]{durrett}. To see this, define the i.i.d random variables $X_i(\gamma)=\log\deg(\theta_i)$ for $\gamma=(\theta_n)_{n\geq1}$, and note that $S_n/n(\gamma):=(X_1(\gamma)+\dots +X_n(\gamma))/n=\log\deg(\gamma_n)^{1/n}$. Here we use that $S$ is a set of morphisms, ensuring that the degree of $\gamma_n$ is the product of its corresponding component degrees. However, we have instead used Birkhoff's Ergodic Theorem since it more closely resembles Kingman's Theorem, our tool for rational maps.  
\end{remark}   
\section{Arithmetic Degrees}\label{sec:arithdeg}
We begin this section with a proof that the dynamical degree bounds (a suitable notion) of the arithmetic growth rate of most orbits. 
\begin{proof}[(Proof of Theorem \ref{thm:arith<dyn})] Let $k\geq1$ be an integer and let $Q\in\mathbb{P}^N(\overline{\mathbb{Q}})_S$. Then a standard triangle inequality estimate (see, for instance, the proof of \cite[Theorem 3.11]{SilvDyn}) implies that \vspace{.1cm}
\begin{equation}\label{rat:bd1} 
h(f(Q))\leq\deg(f) \,h(Q)+C(k,S)\;\;\;\;\; \text{for all $f\in\Phi_{S,k}$}.\vspace{.1cm}  
\end{equation}  
To see this, note that there is such a constant for each $f$ and only finitely many $f$'s, since $S$ is a finite set. In particular, if $P\in \mathbb{P}^N(\overline{\mathbb{Q}})_S$, if $n\geq1$, and if  $\gamma_{nk}=f_n\circ f_{n-1}\circ\dots f_1$ for some $f_i\in \Phi_{S,k}$, then repeated application of the bound in (\ref{rat:bd1}) implies that \vspace{.25cm}  
\begin{equation*} 
\begin{split}
h(\gamma_{nk}(P))\leq&\deg(f_n)\deg(f_{n-1})\dots\deg(f_1)\Scale[.97]{\Big(h(P)+\frac{C(k,S)}{\deg(f_1)}+\frac{C(k,S)}{\deg(f_1)\deg(f_2)}+\dots+\frac{C(k,S)}{\deg(f_1)\dots\deg(f_n)}\Big)} \\[5pt]
\leq&\deg(f_n)\deg(f_{n-1})\dots\deg(f_1) \Big(h(P)+C(k,S)\Big). \vspace{.15cm}  
\end{split} 
\end{equation*}
Here we use our assumption that $S$ is degree independent, so that $\deg(f_i)\geq2$ for all $1\leq i\leq n$. In particular, \vspace{.1cm}  
\begin{equation}\label{rat:bd2} 
h(\gamma_{nk}(P))\leq \deg(f_n)\deg(f_{n-1})\dots\deg(f_1) \,C(k,S,P) \vspace{.15cm} 
\end{equation}
for all $n,k\geq1$, all $\gamma_{nk}\in\Phi_{S,nk}$ and all $P\in\mathbb{P}^N(\overline{\mathbb{Q}})_S$. From here we use Birkhoff's Ergodic Theorem to control the right hand side of (\ref{rat:bd2}) above. Namely, let $T_{(k)}:\Phi_S\rightarrow\Phi_S$ denote the $k$-shift map, $T_{(k)}:=T^k=T\circ T\circ\dots\circ T$. In particular, since the shift map $T$ is ergodic and measure preserving by Lemma \ref{shift}, so is $T_{(k)}$. Likewise, consider the random variable $F_{(k)}:\Phi_S\rightarrow\mathbb{R}_{\geq0}$ given by \vspace{.1cm} 
\[ F_{(k)}(\gamma)=\frac{\log\deg(\gamma_k)}{k}\;\;\;\;\ \text{for $\gamma\in\Phi_S$}.\]
Then, rewriting the bound in (\ref{rat:bd2}) and taking $nk$-th roots, we see that \vspace{.1cm} 
\begin{equation}\label{rat:bd3}
h(\gamma_{nk}(P))^{1/nk}\leq \bigg(\exp\frac{1}{n}\sum_{i=0}^{n-1}F_{(k)}\big(T_{(k)}^i(\gamma)\big)\bigg)
\,C(k,S,P)^{1/nk}.\end{equation}   
In particular, (\ref{rat:bd3}) implies that 
\begin{equation}\label{rat:bd4}
\limsup_{n\rightarrow\infty}h(\gamma_{nk}(P))^{1/nk}\leq\limsup_{n\rightarrow\infty}\bigg(\exp\,\frac{1}{n}\sum_{i=0}^{n-1}F_{(k)}\big(T_{(k)}^i(\gamma)\big)\bigg).
\end{equation} 
However, Birkhoff's Ergodic Theorem \ref{birk} implies that 
\begin{equation}\label{rat:lim}
\lim_{n\rightarrow\infty}\frac{1}{n}\sum_{i=0}^{n-1}F_{(k)}\big(T_{(k)}^i(\gamma)\big)=\mathbb{E}[F_{(k)}]
\end{equation} 
for almost every $\gamma\in\Phi_{S}$. Moreover, since a countable union of measure zero sets has measure zero, we see that the limit in (\ref{rat:lim}) is \textbf{true for all k}, for almost every $\gamma\in\Phi_S$. Therefore, (\ref{rat:bd4}), (\ref{rat:lim}), and the fact that the exponential function is continuous together imply that \vspace{.1cm} 
\begin{equation}\label{rat:bigbd}
\limsup_{n\rightarrow\infty}h(\gamma_{nk}(P))^{1/nk}\leq \exp\bigg(\mathbb{E}\Big[\frac{\log\deg(\gamma_k)}{k}\Big]\bigg) \vspace{.1cm}  
\end{equation}
holds for all $k$, for almost every $\gamma\in\Phi_S$.

We next show that the overall limit supremum in Theorem \ref{thm:arith<dyn} can be computed using the subsequence of multiples of $k$ (for any $k\geq1$). To do this, define constants \vspace{.1cm} 
\begin{equation}\label{rat:degbd} 
d_{S,k}:=\max_{\substack{f\in\Phi_{S,r}\\ 0\leq r<k}}\deg(f)\;\;\; \text{and}\;\;\;\ B_{S,k}:=\max_{0\leq r<k} C(r,S);\vspace{.1cm}
\end{equation} 
here, we remind the reader that $C(r,S)$ is the height bound constant given by \vspace{.1cm}
\begin{equation}\label{rat:degbd2}
C(r,S)=\max_{f\in\Phi_{S,r}}\sup_{Q\in\mathbb{P}^N(\overline{\mathbb{Q}})}\{h(f(Q))-\deg(f)h(Q)\}. \vspace{.1cm}
\end{equation} 
In particular, both $d_{S,k}$ and $B_{S,k}$ are finite since $S$ is a finite set. From here we proceed as in the proof of \cite[Proposition 12]{SilvermanPN}. Namely, for any $k\geq1$ we have that: \vspace{.15cm}   
\begin{equation}\label{rat:subseq}
\begin{split}
\limsup_{m\rightarrow\infty} h(\gamma_m(P))^{1/m}&=\limsup_{n\rightarrow\infty} \max_{0\leq r<k} h(\gamma_{r+nk}(P))^{1/(r+nk)}\\[5pt]
&\leq\limsup_{n\rightarrow\infty} \Big(d_{S,k}\,h(\gamma_{nk}(P))+B_{S,k}\Big)^{1/nk}\;\;\;\;\;\ \text{by (\ref{rat:bd1}),  (\ref{rat:degbd}), and (\ref{rat:degbd2})}\\[5pt] 
&=\limsup_{n\rightarrow\infty} h(\gamma_{nk}(P))^{1/nk} \vspace{.1cm}
\end{split} 
\end{equation} 
Hence, combining the bound in (\ref{rat:bigbd}) with (\ref{rat:subseq}), we see that \vspace{.2cm}
\begin{equation}{\label{rat:bd6}}
\limsup_{m\rightarrow\infty} h(\gamma_m(P))^{1/m}\leq \exp\bigg(\mathbb{E}\Big[\frac{\log\deg(\gamma_k)}{k}\Big]\bigg)=\exp\bigg(\frac{\mathbb{E}[\log\deg(\gamma_k)]}{k}\bigg) \vspace{.2cm}
\end{equation}
holds for all $k\geq1$, for almost every $\gamma\in\Phi_S$. In particular, letting $k\rightarrow\infty$, we deduce from (\ref{kinglim}) and (\ref{rat:bd6}) that \vspace{.1cm}
\[\limsup_{m\rightarrow\infty} h(\gamma_m(P))^{1/m}\leq \delta_{S,\nu}\]
as desired. 
\end{proof} 

We now prove that for sets of morphisms, if $P\in\mathbb{P}^N(\overline{\mathbb{Q}})$ is an almost surely wandering point, then the arithmetic degree of a $P$-orbit exists and is equal to the dynamical degree almost surely. To do this, we use the notion of canonical heights associated to infinite sequences $\gamma\in\Phi_S$ defined in \cite{Kawaguchi} and revisited in \cite{stochastic}. In particular, if $S$ is height controlled, then for all points $P\in\mathbb{P}^N(\overline{\mathbb{Q}})$ and all sequences $\gamma\in\Phi_S$, the limit 
\begin{equation}\label{canht}
\hat{h}_\gamma(P)=\lim_{n\rightarrow\infty}\frac{h(\gamma_n(P))}{\deg(\gamma_n)}
\end{equation} 
exists; see \cite[Theorem 1.1]{Kawaguchi} or \cite[Theorem 1.2]{stochastic}. Moreover, we call $\hat{h}_{\gamma}(P)$ the \emph{canonical height} of the pair $(\gamma,P)$. With this tool in place, we are ready to relate the dynamical and arithmetic degrees for most sequences of elements of $S$.  
\begin{proof}[(Proof of Theorem \ref{thm:arithdeg})] If $P\in\mathbb{P}^N(\overline{\mathbb{Q}})$ is almost surely wandering for a set or morphisms $S$, then $\hat{h}_\gamma(P)>0$ for almost every $\gamma\in\Phi_S$; see \cite[Remark 2.3]{stochastic}. In particular, for each such $\gamma$ we can choose $\epsilon_\gamma>0$ and $m_\gamma$ so that  \vspace{.15cm}
\[\hat{h}_\gamma(P)-\epsilon_\gamma\leq\frac{h((\gamma_n(P))}{\deg(\gamma_n)}\leq \hat{h}_\gamma(P)+\epsilon_\gamma \;\;\;\;\;\text{and}\;\;\;\;\; \hat{h}_\gamma(P)-\epsilon>0  \vspace{.1cm}\]    
for all $n>m_\gamma$. In particular, it follows that  \vspace{.15cm}  
\begin{equation}\label{bd1}
(\hat{h}_\gamma(P)-\epsilon_\gamma)^{1/n}\deg(\gamma_n)^{1/n}\leq h(\gamma_n(P))^{1/n}\leq (\hat{h}_\gamma(P)+\epsilon_\gamma)^{1/n}\deg(\gamma_n)^{1/n}  \vspace{.15cm}
\end{equation}  
for all $n>m_\gamma$. On the other hand, Theorem \ref{thm:dyndeg} implies that $\displaystyle{\lim\deg(\gamma_n)^{1/n}=\delta_{S,\nu}}$ for almost every $\gamma\in\Phi_S$. Hence, Theorem \ref{thm:dyndeg} and (\ref{bd1}) together imply that \vspace{.25cm}   
\begin{equation*}
\begin{split}  
\delta_{S,\nu}=1\cdot\lim_{n\rightarrow\infty}\deg(\gamma_n)^{1/n}=\lim_{n\rightarrow\infty}&(\hat{h}_\gamma(P)-\epsilon_\gamma)^{1/n}\deg(\gamma_n)^{1/n}\leq\alpha_\gamma(P)\\[6pt]  
\text{and} \\[6pt] 
\alpha_\gamma(P)\leq\lim_{n\rightarrow\infty}(\hat{h}_\gamma(P)+\epsilon_\gamma)^{1/n}\deg&(\gamma_n)^{1/n}\leq 1\cdot\lim_{n\rightarrow\infty}\deg(\gamma_n)^{1/n}=\delta_{S,\nu} \vspace{25cm}   
\end{split} 
\end{equation*}
for almost every $\gamma\in\Phi_S$. In particular, we see that $\delta_{S,\nu}=\alpha_{\gamma}(P)$ for such $\gamma\in\Phi_S$.  
\end{proof}  
Combining our knowledge of dynamical and arithmetic degrees for sets of morphisms, we prove an asymptotic orbit counting result. 
\begin{proof}[(Proof of Corollary \ref{cor:arithdeg})] We follow the proof of \cite[Proposition 3]{KawaguchiSilverman}, which we include for completeness. Let $\gamma\in\Phi_S$ be such that  $\alpha_\gamma(P)=\delta_{S,\nu}$, true of almost every $\gamma$ by Theorem \ref{thm:arithdeg}. Then for every $\epsilon>0$ there is an integer $n_0=n_0(\epsilon,\gamma)$ so that 
\[(1-\epsilon)\delta_{S,\nu}\leq h(\gamma_n(P))^{1/n}\leq(1+\epsilon)\delta_{S,\nu}\]
for all $n\geq n_0$. In particular, it follows that \vspace{.1cm} 
\begin{equation*}
\begin{split} 
\{n\geq n_0\,:\,(1+\epsilon)\delta_{S,\nu}\leq B^{1/n}\}&\subset\{n\geq n_0\,:\,h(\gamma_n(P))\leq B\} \\[2pt] 
&\text{and}\\[2pt] 
\{n\geq n_0\,:\,h(\gamma_n(P))\leq B\}&\subset\{n\geq n_0\,:\,(1-\epsilon)\delta_{S,\nu}\leq B^{1/n}\}. 
\end{split} 
\end{equation*}
Therefore, after counting the number elements in the sets above, we see that 
\begin{equation*}
\begin{split} 
\frac{\log(B)}{\log((1+\epsilon)\delta_{S,\nu})}-1&\leq\#\{n\geq0\,:\,h(\gamma_n(P))\leq B\}\\[2pt] 
&\text{and}\\[2pt] 
\#\{n\geq0\,:\,h(\gamma_n(P))\leq B\}&\leq\frac{\log(B)}{\log((1-\epsilon)\delta_{S,\nu})}+n_0+1.  \vspace{.2cm}
\end{split} 
\end{equation*} 
Hence, dividing by $\log(B)$ and letting $B$ tend to infinity gives \vspace{.25cm} 
\begin{equation*}
\begin{split} 
\frac{1}{\log((1+\epsilon)\delta_{S,\nu})}&\leq\liminf_{B\rightarrow\infty}\frac{\#\{n\geq0\,:\,h(\gamma_n(P))\leq B\}}{\log(B)}\\[3pt] 
&\text{and} \\[3pt] 
\limsup_{B\rightarrow\infty}\frac{\#\{n\geq0\,:\,h(\gamma_n(P))\leq B\}}{\log(B)}&\leq \frac{1}{\log((1-\epsilon)\delta_{S,\nu})}.  \vspace{.1cm}
\end{split}  
\end{equation*}  
In particular, since $\epsilon$ was arbitrary, we deduce that   \vspace{.25cm}
\[\lim_{B\rightarrow\infty}\frac{\#\{n\geq0\,:\,h(\gamma_n(P))\leq B\}}{\log(B)}=\frac{1}{\log(\delta_{S,\nu})}\] 
for almost every $\gamma\in\Phi_S$ as claimed. 
\end{proof} 
\begin{remark} For rational maps, one can use Theorem \ref{thm:arith<dyn} to give an asymptotic \emph{upper} bound on the cardinality of the set $\{n:h(\gamma_n(P))\leq B\}$ for initial points $P\in\mathbb{P}^N(\overline{\mathbb{Q}})_S$.  
\end{remark} 
\section{Almost Surely Wandering Points}\label{sec:aswandering}
Let $S$ be a height controlled set of endomorphisms on $\mathbb{P}^N$. Given the (nearly) uniform control over height growth rates that is possible when $P$ is an almost surely wandering point for $S$ (see Theorem \ref{thm:arithdeg} and Corollary \ref{cor:arithdeg} above), it is useful to have an alternative characterization of this property, which we now establish. To do this, we use the expected canonical height function $\mathbb{E}_{\bar{\nu}}[\hat{h}]:\mathbb{P}^N\rightarrow\mathbb{R}$ attached to pairs $(S,\nu)$ defined in \cite[Theorem 1.2]{stochastic}. As we will see, this height function is a useful tool for analyzing the collective action of the maps in $S$ on points in $\mathbb{P}^N$; see, for instance, \cite[Corollary 1.4]{stochastic}. 

To define $\mathbb{E}_{\bar{\nu}}[\hat{h}]$, note first that to any (fixed) point $Q\in\mathbb{P}^N(\overline{\mathbb{Q}})$, we can define a function $\hat{h}_Q:\Phi_S\rightarrow\mathbb{R}$ given by $\hat{h}_Q(\gamma)=\hat{h}_\gamma(Q)$; see (\ref{canht}) above. In particular, since $\Phi_S$ is a probability space, we may view $\hat{h}_Q$ as a random variable and ask about its expected (or average) value: 
\begin{equation}\label{expht}
\mathbb{E}_{\bar{\nu}}[\hat{h}](Q):=\mathbb{E}[\hat{h}_Q]=\int_{\Phi_S}\hat{h}_Q(\gamma)d\nu 
\end{equation} 
In particular, the expected canonical height function $\mathbb{E}_{\bar{\nu}}[\hat{h}]$ satisfies analogs of the usual properties of canonical heights (when $S=\{\phi\}$) due to Call-Silverman; see \cite[Theorem 1.2]{stochastic}. For example, the kernel of $\mathbb{E}_{\bar{\nu}}[\hat{h}]$ is exactly the set of points with finite $S$-orbit, independent of the choice of $\nu$. To state this result formally, we say that $(S,\nu)$ is \emph{strictly positive} if $\nu(\phi)>0$ for all $\phi\in S$. 

\begin{corollary}\label{cor:finiteorbit} Let $(S,\nu)$ be a strictly positive, height controlled collection of endomorphisms on $\mathbb{P}^N$. Then for all $Q\in\mathbb{P}^N(\overline{\mathbb{Q}})$, the following are equivalent: \vspace{.05cm} 
\begin{enumerate}
\item[\textup{(1)}] $Q$ has finite $S$-orbit. \vspace{.15cm} 
\item[\textup{(2)}] $\bar{\nu}\big(\{\gamma\in\Phi_S\,:\, \text{$\Orb_\gamma(Q)$ is finite}\}\big)=1$. \vspace{.15cm}  
\item[\textup{(3)}] $\mathbb{E}_{\bar{\nu}}[\hat{h}](Q)=0.$  
\end{enumerate}    
\end{corollary}
\begin{remark} See \cite[Corollary 1.4]{stochastic} for a more general version of this result. 
\end{remark} 
We are now ready to prove our alternative characterization of almost surely wandering points from the introduction. Namely, $P$ is \emph{not} almost surely wandering if and only if there exists some point $Q\in\Orb_S(P)$ such that $\mathbb{E}_{\bar{\nu}}[\hat{h}](Q)=0$; equivalently, by Corollary \ref{cor:finiteorbit}, $P$ is not almost surely wandering if and only if some point in the $S$-orbit of $P$ has finite $S$-orbit; see Example \ref{eg:halffinite} above for an illustration of this phenomenon.         
\begin{proof}[(Proof of Theorem \ref{thm:a.s})] Suppose that there exists $Q\in\Orb_S(P)$ such that $\mathbb{E}_{\bar{\nu}}[\hat{h}_Q]=0$. Then $Q=\gamma_m(P)$ for some $\gamma_m\in\Phi_{S,m}$ and some $m\geq0$. Now consider the set \vspace{.15cm}
\[U(\gamma_m)=\{\rho\in\Phi_S\,:\,\rho_m=\gamma_m\}.\]
Then, since $(\Phi_S,\bar{\nu})$ is the space of i.i.d sequences of elements of $S$ (i.e., we endow $\Phi_S=\prod_{i=1}^\infty S$ with the natural product measure induced by $\nu$), we see that 
\[\bar{\nu}\big(U(\gamma_m)\big)=\nu_m\big(\{\rho_m\in\Phi_{S,m}\,:\,\rho_m=\gamma_m\}\big)=\prod_{i=1}^m\nu(\theta_i), \;\;\;\text{where}\;\;\gamma_m=\theta_m\circ\theta_{m-1}\circ\dots\circ\theta_1.\]
Therefore, $\bar{\nu}(U(\gamma_m))>0$; here we use that $\nu$ is strictly positive. On the other hand, if $\gamma\in U(\gamma_m)$, then $\Orb_\gamma(P)$ is finite. To see this, note that \vspace{.15cm}   
\begin{equation*}
\begin{split} 
\Orb_\gamma(P)&=\{P,\gamma_1(P),\gamma_2(P),\dots,\gamma_{m-1}(P)\}\cup \Orb_{\,T^m(\gamma)}(Q) \\[5pt] 
&\subseteq \{P,\gamma_1(P),\gamma_2(P),\dots,\gamma_{m-1}(P)\}\cup\Orb_S(Q).  
\end{split} 
\end{equation*} 
But $\Orb_S(Q)$ is finite, since $\mathbb{E}_{\bar{\nu}}[\hat{h}_Q]=0$; see Corollary \ref{cor:finiteorbit} above. Therefore, \vspace{.1cm}  
\[U(\gamma_m)\subseteq \{\gamma\in\Phi_S:\,\Orb_\gamma(P)\; \text{is finite}\}; \vspace{.1cm}\]
hence, the latter set has positive measure. In particular, the point $P$ is not almost surely wandering for $(S,\nu)$ as claimed.   

We now show that if $\mathbb{E}_{\bar{\nu}}[\hat{h}_Q]>0$ for all $Q\in\Orb_S(P)$, then $P$ is almost surely wandering. To do this, we need a few elementary lemmas. However, in order to make clean statements of these auxiliary results, we introduce some additional notation. Let $M_S$ be the monoid of all functions obtained by composing finitely many elements of $S$ (together with the identity function id on $\mathbb{P}^1$): \vspace{.1cm}
\[M_S=\Big\{\rho\in K(x)\,:\;\rho=\text{id}\;\,\text{or}\,\;\rho=(\phi_n\circ\phi_{n-1}\circ\dots \circ\phi_1)\;\; \text{for $\phi_i\in S$, $n\geq1$}\Big\}. \vspace{.1cm} \]
In particular, given a initial point $P\in\mathbb{P}^N(\overline{\mathbb{Q}})$, we may redefine the $S$-\emph{orbit} of $P$ to be \vspace{.15cm} 
\begin{equation}{\label{orbitaldef}}
\Orb_S(P):=\big\{\rho(P)\;:\:\text{for $\rho\in M_S$}\big\};
\end{equation} 
compare to the definition given in the introduction.  
\begin{lemma}[The Escape Lemma]\label{lem:escape} Let $X$ be a set and let $S$ be any collection of self maps on $X$. Suppose that $F=\{Q_1,Q_2, \dots, Q_n\}$ is a finite subset of $X$ with the following properties: \vspace{.2cm} 
\begin{enumerate}
\item[\textup{(1)}] $\phi(X\mysetminus F)\subseteq X\mysetminus F$ for all $\phi\in S$, \vspace{.2cm} 
\item[\textup{(2)}] For each $Q_i\in F$, there exists $f_i\in M_S$ such that $f_i(Q_i)\not\in F$. \vspace{.2cm}
\end{enumerate}  
Then there exists $g\in M_S$ such that $g(Q_i)\not\in F$ for all $Q_i\in F$. 
\begin{remark} We call a function $g$ as in Lemma \ref{lem:escape} an \emph{escape function} for the subset $F$.  
\end{remark} 
\end{lemma} 
\begin{proof}[(Proof of Lemma \ref{lem:escape})] Let $g_1=f_1$, and for $m>1$ define 
\[
\begin{cases} 
g_m=g_{m-1} & \text{if $g_{m-1}(Q_m)\not\in F$},\\[3pt] 
g_m=f_{j_m}\circ g_{m-1} & \text{if $g_{m-1}(Q_m)\in F$, where $g_{m-1}(Q_m)=Q_{j_m}$}.
\end{cases}
\]
Then it is easy to check that $g=g_n$ is an escape function for $F$.   
\end{proof}  
In order to apply the escape lemma above to prove the final direction of Theorem \ref{thm:a.s}, we need a few facts pertaining to heights. To state these facts, let $C_S:=\max\big\{\sup_{\phi\in S} C(\phi),1\big\}$; for a definition of $C(\phi)$ see (\ref{htconstant}) above. Note that $C_S$ is well defined, since $S$ is finite.  
\begin{lemma}\label{lem:stablecomplement} Let $K/\mathbb{Q}$ be a finite extension over which every map in $S$ is defined, let $P\in \mathbb{P}^N(K)$, and let 
\[F_{P,S}:=\{Q\in \Orb_S(P)\,:\, h(Q)\leq 2C_S\}.\] 
Then 
\[\phi\big(\Orb_S(P)\mysetminus F_{P,S}\big)\subseteq \Orb_S(P)\mysetminus F_{P,S},\vspace{.5cm}\]
for all $\phi\in S$. Equivalently, the complement of $F_{P,S}$ in $\Orb_S(P)$ is $S$-stable. 
\end{lemma}
\begin{proof}[(Proof of Lemma \ref{lem:stablecomplement})] If $Q\in\Orb_S(P)\mysetminus F_{P,S}$, then $h(Q)>2C_S$. On the other hand, since $h(\phi(Q))\geq d_\phi h(Q)-C_S$ for all $\phi\in S$ (by definition of $C_S$) and $d_\phi\geq2$ by assumption, we see that \vspace{.1cm}  
\[ h(\phi(Q))\geq 2h(Q)-C_S>3C_S>2C_S\]
for all $\phi\in S$. In particular, the complement of $F_{P,S}$ in $\Orb_S(P)$ is $S$-stable as claimed; here we use also that $\Orb_P(S)$ is $S$-stable by the definition in (\ref{orbitaldef}) above.   
\end{proof} 
As a final preliminary step, we note the following fact.  
\begin{lemma}\label{lem:hatzero} If $\hat{h}_\gamma(P)=0$, then $\Orb_\gamma(P)\subseteq F_{P,S}$.  
\end{lemma} 
\begin{proof}[(Proof of Lemma \ref{lem:hatzero})] We prove the contrapositive. Write $\gamma=(\theta_i)_{i=1}^\infty$ for some $\theta_i\in S$ and suppose that $\Orb_\gamma(P)\not\subseteq F_{P,S}$, i.e., there exists an integer $m\geq0$ and $Q=\gamma_m(P)$ such that $h(Q)>2C_S$. Now let $d=\min_{\phi\in S}d_\phi$ so that $d\geq2$ by assumption. Then an adaptation of Tate's telescoping argument implies that for any $Q\in\mathbb{P}^N(\overline{\mathbb{Q}})$ and any map $\rho_r\in M_S$ of length $r$, we have that \vspace{.1cm}    
\[h(\rho_r(Q))\geq \deg(\rho_r)\Big(h(Q)-\frac{C_S}{(d-1)}\Big)\geq\deg(\rho_r)(h(Q)-C_S);\]
see \cite[Lemma 2.1]{stochastic}. Here we use that $d\geq2$. In particular, applying the bound above to the map $\rho_{n-m}=\theta_n\circ\theta_{n-1}\dots \theta_{m+1}$ and the point $Q=\gamma_m(P)$, we see that \vspace{.1cm}  
\[h(\gamma_n(P))\geq \deg(\rho_{n-m}) (h(Q)-C_S)\geq \deg(\rho_{n-m})C_S.\vspace{.1cm}\]
Therefore, dividing both sides by $\deg(\gamma_n)$ we see that \vspace{.1cm} 
\[\frac{h(\gamma_n(P))}{\deg(\gamma_n)}\geq \frac{1}{\deg(\gamma_m)}C_S. \vspace{.1cm} \]
In particular, letting $n$ tend to infinity (note that the right-hand-side of the bound above does not depend on $m$), we see that $\hat{h}_\gamma(P)>0$ as claimed. 
\end{proof} 
We now return to the proof of Theorem \ref{thm:a.s}. Assume that $\mathbb{E}_{\bar{\nu}}[\hat{h}](Q)>0$ for all points $Q$ in the $S$-orbit of $P$, and choose a finite extension $K/\mathbb{Q}$ such that the initial point $P\in\mathbb{P}^1(K)$ and the maps in $S$ are all defined over $K$. Note that $X=\Orb_S(P)$ and $F=F_{P,S}$ (a finite set by Northcott's Theorem) satisfy condition (1) of Lemma \ref{lem:escape} by Lemma \ref{lem:stablecomplement}. Moreover, condition (2) of Lemma \ref{lem:escape} is also satisfied. For if (2) fails, then there exists $Q\in F_{P,S}$ such that $\Orb_S(Q)\subseteq F_{P,S}$; hence $\mathbb{E}_{\bar{\nu}}[\hat{h}](Q)=0$ by \cite[Corollary 1.4]{stochastic}, contradiction. In particular, Lemma \ref{lem:escape} implies that there exists an escape function $g=g_{S,P}$ for $F$. Now consider the set 
\[I_{S,P}=\Big\{(\theta_i)_{i=1}^\infty\in\Phi_S\,:\,\theta_n\circ\theta_{n-1}\dots\circ\theta_m=g\;\, \text{for some $n>m$}\Big\};\]
in words, $I_{S,P}$ is the set of sequences where the finite string $g$ appears at some point. In particular, it follows from the well known ``Infinite Monkey Theorem" that $\bar{\nu}(I_{S,P})=1$; in fact, the (smaller) set of sequences where $g$ appears infinitely often already has full measure by Kolmogorov zero-one law; see \cite[Theorem 10.6]{ProbabilityText}. 

Finally, we will show that if $\gamma\in I_{S,P}$ then $\hat{h}_\gamma(P)>0$. To see this, write $\gamma_n=g\circ\gamma_m$ for some $n>m$. If $\gamma_m(P)\not\in F_{S,P}$, then Lemma \ref{lem:hatzero} implies that $\hat{h}_\gamma(P)>0$. On the other hand, if $\gamma_{m}(P)\in F_{S,P}$, then $g(\gamma_m(P))\not\in F_{P,S}$ by definition of $g$; see Lemma \ref{lem:escape}. Therefore, $\gamma_n(P)=g(\gamma_m(P))\not\in F_{S,P}$, and $\hat{h}_\gamma(P)>0$ as well. In particular, 
\begin{equation}{\label{eq:measure0}} 
\bar{\nu}(\{\gamma\in\Phi_S: \hat{h}_\gamma(P)=0\})=0.
\end{equation}
However, \cite[Theorem 1.1]{Kawaguchi} implies that $\hat{h}_\gamma(P)=0$ if and only if $\Orb_\gamma(P)$ is finite. Therefore, (\ref{eq:measure0}) implies that 
\[ \bar{\nu}\big(\{\gamma\in\Phi_S:\,\Orb_\gamma(P)\; \text{is finite}\}\big)=0.\]
Hence, $P$ is almost surely wandering for $(S, \nu)$ as claimed.     
\end{proof} 
    
\end{document}